\documentclass[12pt,centertags,oneside, reqno]{amsart}
\usepackage{ifpdf}

\usepackage{amssymb}
\usepackage{braket}
\usepackage[all]{xy}
\usepackage{cancel}

\usepackage{amscd,amsxtra,calc}
\usepackage{cmmib57}
\usepackage{url}
\usepackage{ulem}
\usepackage{xcolor}
\usepackage[a4paper,width=16.2cm,top=3cm,bottom=3cm]{geometry}

\numberwithin{equation}{section}

\theoremstyle{plain}
\newtheorem{thm}{Theorem}[section]
\newtheorem{theorem}[thm]{Theorem}

\newtheorem{lemma}[thm]{Lemma}

\newtheorem{proposition}[thm]{Proposition}
\theoremstyle{definition}
\newtheorem{remark}[thm]{Remark}

\newtheorem{definition}[thm]{Definition}

\newtheorem{example}[thm]{Example}

\newtheorem{conjecture}[thm]{Conjecture}

\numberwithin{equation}{section}

\newcommand{\sE}{{\mathcal E}}

\newcommand{\sO}{{\mathcal O}}

\newcommand{\C}{{\mathbb C}}

\renewcommand{\P}{{\mathbb P}}
\newcommand{\Q}{{\mathbb Q}}
\newcommand{\R}{{\mathbb R}}

\newcommand{\Z}{{\mathbb Z}}

\newcommand{\id}{{\rm id\hspace{.1ex}}}

\newcommand{\Aut}{{\rm Aut\hspace{.1ex}}}

\title [Variety of Ueno type]{Endomorphisms of a variety of Ueno type and Kawaguchi-Silverman Conjecture}

\author{Keiji Oguiso}
\address{Mathematical Sciences, the University of Tokyo, Meguro Komaba 3-8-1, Tokyo, Japan, and National Center for Theoretical Sciences, Mathematics Division, National Taiwan University, 
Taipei, Taiwan}
\email{oguiso@ms.u-tokyo.ac.jp}

\thanks{The author is partially supported by JSPS Grant-in-Aid (A) 15H05738, JSPS Grant-in-Aid (B) 15H03611 and NCTS scholar program.}

\begin{document}

\maketitle

\begin{abstract}
We first show that the monoid of separable surjective self morphisms of a variety of Ueno type coincides with the group of automorphisms. We also give an explicit description of the automorphism group. As applications, we confirm Kawaguchi Silverman Conjecture for automorphisms of a variety of Ueno type and some Calabi-Yau threefolds, defined over $\overline{\Q}$. 
\end{abstract}

\section{Introduction}\label{sect1}

Throughout this paper, $K = \overline{K}$ is an algebraically closed field of characteristic $p \not= 2, 3$, i.e., $p=0$ or $p \ge 5$. We are mainly interested in the cases $K = \overline{\Q}$ and $\C$ for applications. We set $\zeta_k := {\rm exp}\, (2\pi\sqrt{-1}/k)$ for $k \in \Z_{\ge 1}$ when $K = \overline{\Q} (\subset \C)$ and $\C$. Our main results are Theorems \ref{mainthm1}, \ref{mainthm2}, \ref{mainthm3}, \ref{thm42}.

\subsection{Varieties of Ueno type and the main theorem}\label{subsect11}  
\hfill

We define varieties of Ueno type, which play the most crucial roles in this paper:

\begin{definition}\label{def1} Let $n \ge 2$ be an integer and $\zeta_k \in K$ a primitive $k$th root of $1$ with $k \in \{2,3,4\}$.
\begin{enumerate}

\item Let $E_2 = E_{2, \lambda}$ be the elliptic curve defined by the Weierstass equation 
$$y^2 = x(x-1)(x - \lambda)\,\, {\rm where}\,\,  \lambda \in K \setminus\{0, 1\}.$$
Define $\sigma_2 \in \Aut(E_2)$ of order $2$ by
$$\sigma_2 : (x, y) \mapsto (x, -y).$$
We set 
$$A_{2}^n := E_2^{\times n}\,\, ,\,\, \sigma_{2}^n := \sigma_2 ^ {\times n} \in \Aut (A_{2}^n),$$ 
the self product of $E_2$ of $n$ times and the automorphism of $A_2^n$ of order $2$ given by the product of $\sigma_2$. Let 
$$\pi_{2}^{n} : A_{2}^n \to \overline{X}_{2}^{n}  := A_{2}^{n}/\langle \sigma_{2}^{n} \rangle$$
be the quotient morphism and
$$\nu_{2}^{n} : X_{2}^{n} \to \overline{X}_{2}^{n}$$
be the blow up of $\overline{X}_{2}^{n}$ at the maximal ideals of the singular points ($4^n$ points), which are of type $\frac{1}{2}(1, \ldots, 1)$. Then $X_{2}^{n}$ is a smooth projective variety of dimension $n$. 

\item Let $E_4$ be the elliptic curve defined by the Weierstass equation 
$$y^2 = x(x^2-1).$$ 
Define $\sigma_4 \in \Aut(E_4)$ of order $4$ by
$$\sigma_4 : (x, y) \mapsto (-x, \zeta_4y).$$
We set 
$$A_{4}^n := E_4^{\times n}\,\, ,\,\, \sigma_{4}^n := \sigma_4 ^ {\times n} \in \Aut (A_{4}^n),$$ 
the self product of $E_4$ of $n$ times and the automorphism of $A_4^n$ of order $4$ given by the product of $\sigma_4$. Let 
$$\pi_{4}^{n} : A_{4}^n \to \overline{X}_{4}^{n}  := A_{4}^{n}/\langle \sigma_{4}^{n} \rangle$$
be the quotient morphism and
$$\nu_{4}^{n} : X_{4}^{n} \to \overline{X}_{4}^{n}$$
be the blow up of $\overline{X}_{4}^{n}$ at the maximal ideals of the singular points, which are of type $\frac{1}{4}(1, \ldots, 1)$ ($2^n$ points) or of type $\frac{1}{2}(1, \ldots, 1)$ ($(4^n -2^n)/2$ points) . Then $X_{4}^{n}$ is a smooth projective variety of dimension $n$. 

\item Let $E_3 = E_6$ be the elliptic curve defined by the Weierstass equation 
$$y^2 = x^3-1.$$ 
Define $\sigma_3 \in \Aut(E_3)$ of order $3$, $\sigma_6 \in \Aut(E_3)$ of order $6$ by
$$\sigma_3 : (x, y) \mapsto (\zeta_3 x, y)\,\, ,\,\, \sigma_6 : (x, y) \mapsto (\zeta_3 x, -y).$$
We set 
$$A_{3}^n = A_{6}^{n}:= E_3^{\times n}\,\, ,\,\, \sigma_{3}^n := \sigma_3 ^ {\times n} \in \Aut (A_{3}^n)\,\, ,\,\, \sigma_{6}^n := \sigma_6 ^ {\times n} \in \Aut (A_{6}^n),$$  
the self product of $E_3$ of $n$ times and the automorphisms of $A_3^n$ and $A_6^n$ of order $3$ and of order $6$ given by the product of $\sigma_3$ and $\sigma_6$. Let 
$$\pi_{3}^{n} : A_{3}^n \to \overline{X}_{3}^{n}  := A_{3}^{n}/\langle \sigma_{3}^{n} \rangle$$
be the quotient morphism and
$$\nu_{3}^{n} : X_{3}^{n} \to \overline{X}_{3}^{n}$$
be the blow up of $\overline{X}_{3}^{n}$ at the maximal ideals of the singular points, which are of type $\frac{1}{3}(1, \ldots, 1)$ ($3^n$ points). Similarly, we let 
$$\pi_{6}^{n} : A_{6}^n \to \overline{X}_{6}^{n}  := A_{6}^{n}/\langle \sigma_{6}^{n} \rangle$$
be the quotient morphism and
$$\nu_{6}^{n} : X_{6}^{n} \to \overline{X}_{6}^{n}$$
be the blow up of $\overline{X}_{6}^{n}$ at the maximal ideals of the singular points, which are of type either $\frac{1}{6}(1, \ldots, 1)$ ($1$ point), $\frac{1}{3}(1, \ldots, 1)$ ($(3^n -1)/2$ points).
 or $\frac{1}{2}(1, \ldots, 1)$ ($(4^n -1)/3$ points) . Then $X_{3}^{n}$ and $X_{6}^n$ are smooth projective varieties of dimension $n$. 
\end{enumerate}
We call $X_k^n$ ($k \in \{2, 3, 4, 6\}$) constructed here a {\it variety of Ueno type} if $n \ge 3$. 
\end{definition}

\begin{remark} \label{rem1} 
\begin{enumerate}
\item Let $n=1$. Then $X_k^1$ are $\P^1$ for all $k \in \{2,3,4,6\}$.
\item Let $n=2$. Then $X_2^2$ is nothing but the Kummer K3 surface ${\rm Km}\, (E_2 \times E_2)$, while $X_4^2$, $X_3^2$, $X_6^2$ are rational surfaces. One can check the rationality by using Castelnouvo's criterion of rationality of surfaces or we may also use \cite[Theorem 2]{KL09}.
\item Let $n \ge 3$. Then the Kodaira dimension of $X_{k}^n$ is $\kappa (X_k^n) = -\infty$ exactly when 
$$(n, k) = (3, 4), (3, 6), (4, 6), (5, 6),$$
while $\kappa (X_k^n) = 0$ otherwise. One can check this directly by the canonical bundle formula or again by \cite[Theorem 2]{KL09}.
\end{enumerate}
\end{remark}
Varieties of Ueno type are intensively studied by \cite[Section 16]{Ue75} in connection with the existence of minimal models in pre-minimal model problem era. Varieties of Ueno type are very special but they also provide several concrete examples of smooth projective varieties with remarkable birational-geometric and algebro-dynamical properties as in the following remark:

\begin{remark} \label{rem2} Let $K = \C$ and $n \ge 3$. 
\begin{enumerate}
\item $X_3 := X_3^3$ is a rigid Calabi-Yau threefold (\cite{Be83}). $X_3$ and $X_7$ (see Subsection \ref{subsect41} for the definition of $X_7$) play important roles in the classification of Calabi-Yau threefolds, in the sense that they are exactly two Calabi-Yau threefolds with a birational $c_2$-contraction \cite{OS01}. $X_3$ is the unique underlying Calabi-Yau threefold with an abelian fibration with a relative automorphism of positive entropy \cite{Og24}. $X_3^3$ is only one known example of a Calabi-Yau threefold admitting a primitive automorphism of positive entropy \cite{OT15}. See also \cite{Be83}, \cite{LO09} and so on for other interesting properties of $X_3$.

\item $X_k^n$ is of Kodaira dimension $-\infty$ (Remark \ref{rem1} (3)) if and only if $X_k^n$ is uniruled if and only if $X_k^n$ is rationally connected. See \cite[Theorem 2]{KL09} for the first equivalence and see \cite[Corolarry 4.6 and its proof]{Og19} for the second equivalence whose proof is based on algebraic dynamics. Moreover, it is also shown that $X_6^3$ and $X_4^3$ are rational by \cite{OT15}, \cite{CT15}, and that $X_6^4$ is unirational by \cite{COV15}. Again these $X_k^n$ and ${\rm Hilb}^n\,(S)$ for some smooth projective rational surface $S$ with an automorphism of positive entropy are only known examples of a smooth rationally connected variety with a primitive automorphism of positive entropy (\cite{Og19} and references therein). See also \cite[Subsection 4.2]{DLOZ22} for other interesting dynamical properties on $X_k^n$ in this range. 

\item It is interesting to see if $X_{6}^{4}$ and $X_{6}^{5}$ are rational or not. It is also very interesting to see if a smooth projective rationally connected variety $X_{6}^{5}$ is actually unirational or not. See \cite{Ko96} about basics and several open problems on these three basic classes of varieties; rational, unirational, rationally connected. 
\end{enumerate}
\end{remark}

For a normal projective variety $V$, we denote by ${\rm End}\, (V)$ the semi-group consisting of a separable surjective self morphism $f : V \to V$ of $V$. We say that $f$ is separable if the finite field extension $K(V) \to K(V)$ induced from $f^*$ is separable, which is automatical if $K$ is of characteristic $0$. Note that $\Aut (V) \subset {\rm End}\, (V)$. We call an element of ${\rm End}\, (V)$ an endomorphism of $V$. 

Our main theorem of this paper is the following:

\begin{theorem}\label{mainthm1} Let $X_k^n$ be a variety of Ueno type with $n \ge 3$. Then
$${\rm End}\, (X_k^n) = \Aut (X_k^n).$$
Moreover, 
$$\Aut (X_2^n) = (E_2^{\times n})^{\langle \sigma_2^n \rangle} \rtimes \Aut_{{\rm group}}(E_2^{\times n})/\langle \sigma_2^n \rangle \simeq (\Z/2 \oplus \Z/2)^{\oplus n} \rtimes {\rm GL}\, (n, \Z[\zeta_k])/\langle \sigma_2^n \rangle,$$
where $k=4$ if $E_2 \simeq E_4$, $k=3$ if $E_2 \simeq E_3 = E_6$ and $k=2$ otherwise, 
$$\Aut (X_4^n) = (E_4^{\times n})^{\langle \sigma_4^n \rangle} \rtimes \Aut_{{\rm group}}(E_4^{\times n})/\langle \sigma_4^n \rangle \simeq (\Z/2)^{\oplus n} \rtimes {\rm PGL}\, (n, \Z[\zeta_4]), $$
$$\Aut (X_3^n) = (E_3^{\times n})^{\langle \sigma_3^n \rangle} \rtimes \Aut_{{\rm group}}(E_3^{\times n})/\langle \sigma_3^n \rangle \simeq (\Z/3)^{\oplus n} \rtimes {\rm GL}\, (n, \Z[\zeta_3])/\langle \sigma_3^n \rangle, $$
$$\Aut (X_6^n) = (E_6^{\times n})^{\langle \sigma_6^n \rangle} \rtimes \Aut_{{\rm group}}(E_6^{\times n})/\langle \sigma_6^n \rangle \simeq {\rm PGL}\, (n, \Z[\zeta_6]).$$
\end{theorem}

\begin{remark}\label{rem3} Consider the case $n=2$. 
\begin{enumerate}
\item As in Theorem \ref{mainthm1}, ${\rm End}\, (X_2^2) = \Aut (X_2^2)$, as $X_2^2 = {\rm Km}\, (E_2 \times E_2)$ is a Kummer K3 surface (cf. Remark \ref{rem22} (1)). I wonder if ${\rm End}\, (X_k^2) = \Aut (X_k^2)$ or not when $k=3$, $4$, $6$.
\item There are many elements of $\Aut (X_k^2)$ which do not come from $\Aut (A_k^2)$. This is a big difference between the case $n \ge 3$ in Theorem \ref{mainthm1} and the case $n=2$.
\end{enumerate} 
\end{remark} 

We shall give a unified proof of Theorem \ref{mainthm1} for all $n \ge 3$, $k \in \{2,3,4, 6\}$ and all $K$ in Section \ref{sect2}, while the most interesting cases for our application are the cases where $(k, n) = (3, 3)$ and $\kappa (X_k^n) = -\infty$ over $K = \overline{\Q}$ and $\C$. 

\medskip

We then give two applications of Theorem \ref{mainthm1} and its proof.

\subsection{Kawaguchi-Silverman Conjecture (KSC) for a variety of Ueno type}\label{subsect12} 
\hfill 

Let $V$ be a normal projective variety defined over $\overline{\Q}$ and $f \in {\rm End}\, (V)$. Unless stated otherwise, $x \in V$ means that $x$ is a closed point of $V$ or equivalently, $x \in V(\overline{\Q})$. Kawaguchi-Silverman Conjecture (KSC) for $f$ asks a relation between the dynamical degree $d_1(f)$, which is one of the most basic complex-dynamical or algebro-dynamical invariants that measures algebro-geometric complexity of $f$, and the arithmetic degree function $a_f(x)$ ($x \in V$), which is an arithmetic dynamical invariant of $f$ that measures arithmetic complexity of $f$ defined over $\overline{\Q}$. See \cite[Conjecture 6]{KS16b} for the original statement and excellent surveys \cite{Ma23}, \cite{MZ23} and references therein for current status on KSC (See also Subsection \ref{subsect31} for a brief account relevant to us). It is known that $$1 \le a_f(x) \le d_1(f)$$ 
for all $x \in V$ when $f \in {\rm End}\, (V)$ (\cite[Theorem 3]{KS16a} and \cite{Ma20}). For $x \in V$, we denote by 
$$O_f(x) := \{f^n(x)\,|\, n \in \Z_{\ge 0}\}$$ 
the forward orbit of $x \in V$ under $f$. KSC for $f \in {\rm End}\, (V)$ is the following \cite[Conjecture 6]{KS16b}:

\begin{conjecture}\label{conj1} Let $V$ be a normal projective variety defined over $\overline{\Q}$ and $f \in {\rm End}\, (V)$. Let $x \in V$. Then, if $O_f(x)$ is Zariski dense in $V$, then $a_f(x) = d_1(f)$. 
\end{conjecture}

Conjecture \ref{conj1} obviously holds if $d_1(f) = 1$ or if there is no $x \in V$ with Zariski dense $O_f(x)$. As in \cite{CLO22}, we say that KSC holds {\it non-trivially and non-vacuously} for ${\rm End}\, (V)$ if KSC holds for all elements of ${\rm End}\, (V)$ and there are $f \in {\rm End}\, (V)$ and $x \in V$ such that $d_1(f) > 1$ and $O_f(x)$ is Zariski dense in $V$.

As the first application of Theorem \ref{mainthm1} and its proof, we show the following:

\begin{theorem}\label{mainthm2} Let $X_k^n$ be a variety of Ueno type over $\overline{\Q}$ with $n \ge 3$. Then KSC holds for ${\rm End}\, (X_k^n) = {\rm Aut}\, (X_k^n)$ non-trivially and non-vacuously.
\end{theorem}

We give a uniform proof of Theorem \ref{mainthm2} for $n \ge 3$ in Subsection \ref{subsect32} by reducing to KSC for an abelian variety due to Kawaguchi-Silverman and Silverman (\cite[Theorem 4]{KS16a}, \cite[Theorem 1.2]{Si17}), via Theorem \ref{mainthm1} and its proof. It is also worth noticing that Meng and Zhang show that KSC holds for endomorphisms of a {\it non-resolved} $\Q$-abelian variety which is by definition a normal projective variety which has an \'etale in codimension one covering by an abelian variety \cite[Theorem 2.8]{MZ22}.

The most interesting case in Theorem \ref{mainthm2} is the case where $X_k^n$ is rationally connected. KSC for ${\rm End}\, (V)$ is known to be true for any smooth projective rationally connected variety $V$ with inter-amplified endomorphism $f_0$ by \cite[Theorem 1.1]{MY22} (see also \cite{MZ22} and \cite{Yo21}). As necessarily $f_0 \not\in \Aut (V)$, our $X_k^n$ has no inter-amplified endomorphism by Theorem \ref{mainthm1}. 

\subsection{KSC for a Calabi-Yau threefold with $c_2$-contraction}\label{subsect13}  
\hfill 

By a complex Calabi-Yau threefold, we mean a smooth projective threefold defined over $\C$ with $\sO_X(K_X) \simeq \sO_X$ and the underlying analytic space $X^{{\rm an}}$ is simply-connected, i.e., $\pi_1(X^{\rm an}) = 1$. We call $X$ defined over $\overline{\Q}$ a Calabi-Yau threefold, if $X \times_{{\rm Spec}\, \overline{\Q}} {\rm Spec}\, \C$ is a complex Calabi-Yau threefold. 

Note that ${\rm End}\, (X) = \Aut (X)$ for a Calabi-Yau threefold $X$ (cf. Remark \ref{rem22} (1)). Let $K = \overline{\Q}$. KSC for ${\rm End}\, (X) = \Aut (X)$ of a Calabi-Yau threefold $X$ is only interested in the case where $\Aut (X)$ is an infinite group, as otherwise $d_1(f) =1$ for all $g \in \Aut\, (X)$. In such a case, we have a non-zero nef real class $x$ with $(c_2(X).x)_X = 0$ (\cite{Wi97}, see also \cite{Mi87} for pseudo-effectivity of $c_2(X)$). 

We call a surjective morphism $f : X \to W$ from a Calabi-Yau threefold to a normal projective variety $W$ with connected fibers a (non-trivial) $c_2$-contraction if $(c_2(X).f^*H)_X = 0$ for some ample Cartier divisor $H$ on $W$ and $\dim W >0$, possibly $f$ is birational. (\cite{OS01}, see also \cite{Og24} for updated description). So, KSC for $\Aut (X)$ of a Calabi-Yau threefolds $X$ over $\overline{\Q}$ with $c_2$-contraction is of interest. 

Actually \cite[Theorem 1.8]{LS21} considered KSC for Calabi-Yau threefolds $X$ with $c_2$-contraction and asserted the first statement of the following:

\begin{theorem}\label{mainthm3} Let $X$ be a Calabi-Yau threefold defined over $\overline{\Q}$ with a $c_2$-contraction. Then KSC holds for ${\rm Aut}\, (X)$. Moreover, KSC holds non-trivially and non-vacuously if $X$ has a birational $c_2$-contraction. 
\end{theorem}

Unfortunately, proof of \cite[Theorem 1.8]{LS21} is based on a {\it claim that $\Aut (X)$ is finite when $X$ has a birational $c_2$-contraction}, which clearly contradicts to the second statement of Theorem \ref{mainthm3}. Actually, this paper is motivated to find a better understanding of \cite[Theorem 1.8]{LS21}. The second statement is based on the complete description of automorphism groups of Calabi-Yau threefolds with birational $c_2$-contraction (Theorem \ref{thm42}), which may have its own interest. We prove Theorem \ref{thm42} in Subsection \ref{subsect41} and Theorem \ref{mainthm3} in Subsection \ref{subsect42}.

{\bf Acknowledgements.} I would like to express my thanks to Professors Yohsuke Matsuzawa, Hei Fu, Sheng Meng and Sichen Li for their interest in this work and valuable comments. I would like to express my thanks to the referee for his/her careful reading and many corrections.

\section{Proof of Theorem \ref{mainthm1}}\label{sect2}

The goal of this section is to prove Theorem \ref{mainthm1}.

For a normal projective variety $V$, we denote by $N^1(V)$ the free $\Z$-module consisting of the numerical equivalence classes of Cartier divisors on $V$. We denote $N^1(V)_{\Q} := N^1(V) \otimes_{\Z} \Q$, $N^1(V)_{\R} := N^1(V) \otimes_{\Z} \R$ and so on. Unless stated otherwise, curves and subvarieties of $V$ are proper ones.    

The following general lemma should be well-known. Our proof here is essentially the same as in \cite[Lemma 2.3]{Fu02}:  

\begin{lemma}\label{lem21} Let $V$ be a normal projective variety and $f \in {\rm End}\, (V)$, i.e., $f : V \to V$ is a separable surjective self morphism of $V$. Then 
\begin{enumerate}
\item The natural homorphism $f^* : N^1(V)_{\Q} \to N^1(V)_{\Q}$ is 
an isomorphism.
\item $f$ is a finite morphism.
\item If in addition that the canonical divisor $K_V$ is $\Q$-linearly equivalent to $0$, then $f$ is \'etale in codimension one, that is, there is a closed algebraic set $S \subset V$ of codimension $\ge 2$ such that $f|_{V \setminus f^{-1}(S)} : V \setminus f^{-1}(S) \to V \setminus S$ is \'etale.
\end{enumerate}
\end{lemma}
\begin{proof} Since $V$ is projective, for each irreducible projective curve $C \subset V$, there are a curve $C'$ and a positive integer $m$ such that $f_*(C') = mC$ as one cycles. Thus $f^*$ is injective by the projection formula. Indeed, if $f^*D = 0$ in $N^1(V)_{\Q}$, then 
$$0 = (f^*D.C')_V = (D.f_{*}C')_{V} = m(D.C)_{V}\,\, {{\rm hence}}\,\, (D.C)_V = 0$$ 
for all irreducible curves $C \subset V$. Thus $D = 0$ in $N^1(V)_{\Q}$ by definition. Since $N^1(V)_{\Q}$ is a $\Q$-vector space of finite dimension, the linear map $f^*$ is then an isomorphism. This shows (1).

Since $V$ is projective, if $f$ is not finite, then there is an irreducible curve $C \subset V$ such that $f(C)$ is a point. Then $(f^*D.C)_V = 0$ for all $D \in N^1(V)_{\Q}$ by the projection formula. On the other hand, since $V$ is projective, there is a hyperplane class $H \in N^1(V)_{\Q}$ and we have $(H.C)_V > 0$, a contradiction to the surjectivity of $f^*$ in (1). This proves (2). 

Let us show (3). Let 
$$V^{0} := V \setminus (f({\rm Sing}\, V) \cup {\rm Sing}\, V)),$$ 
$$V^{00} := V \setminus f^{-1}(f({\rm Sing}\, V) \cup {\rm Sing}\, V)).$$ 
Since $V$ is normal and $f$ is finite and surjective, $V\setminus V^{0}$ and $V\setminus V^{00}$ are closed subsets of codimension $\ge 2$ in $V$ and the morphism $f^{00} : V^{00} \to V^{0}$ induced from $f$ is a finite surjective morphism. Here $V^{00}$ and $V^{0}$ are also smooth. As $K_V$ is $\Q$-linearly equivalent to $0$, it follows that there is a regular nowhere vanishing pluricanonical form $\omega$ on $V^{0}$. Then $(f^{00})^*\omega$ is a regular pluricanonical form and it is not identically $0$ on $V^{00}$ as $f$ is separable. Then $(f^{00})^*\omega$ is also nowhere vanishing on $V^{00}$. In particular, the Jacobian divisor of $f^{00}$ has to be zero. Since $V^{00}$ and $V^{0}$ are smooth, it follows that $f^{00}$ is \'etale. This proves (3). 
\end{proof}

\begin{remark}\label{rem22} 
\begin{enumerate}
\item If $X$ is a smooth projective variety whose $K_X$ is $\Q$-linearly equivalent to $0$, then any $f \in {\rm End}\, (X)$ is \'etale by the proof of (3), as observed in [Fu, Lemma 2.3]. 
\item We use the notation in Definition \ref{def1}. The multiplication morphism by $2$ on $A_2^n$ ($n \ge 2$) induces a finite surjective morphism $[2] : \overline{X}_2^{n} \to \overline{X}_2^{n}$. The morphism $[2]$ is not \'etale even though $2K_{\overline{X}_2^{n}}$ is linearly equivalent to $0$. One can also check directly that $[2]$ does not lifts to ${\rm End}\, (X_{2}^{n})$, which is also a consequence of Theorem \ref{mainthm1} when $n \ge 3$.   
\end{enumerate}
\end{remark}

From now, let $n \ge 3$ and $X := X_k^n$ be a variety of Ueno type. For ease of notation, we also set 
$$E := E_k\,\, ,\,\, \sigma := \sigma_k^n\,\, ,\,\, A := A_k^n\,\, ,\,\, \overline{X} := \overline{X}_k^n\,\, ,\,\, \pi := \pi_k^n\,\, , \,\, \nu := \nu_k^n.$$Let ${\sE}_i$ ($1 \le i \le m$) be the irreducible components of the exceptional set ${\rm Exc}\, (\nu)$ of $\nu$. Then 
$${\rm Exc}\, (\nu) = \cup_{i=1}^{m} {\sE}_i\,\, ,\,\, {\sE}_i \simeq \P^{n-1}$$ 
and $P_i := \nu({\sE}_i)$ is a point of $\overline{X}$ and 
$${\rm Sing}\, \overline{X} = \{P_i\}_{i=1}^{m}.$$

The following elementary observation is crucial in our proof of Theorem \ref{mainthm1}:

\begin{lemma}\label{lem22} Let $\varphi : \P^{n-1} \to X$ be a non-constant morphism. Then $\varphi(\P^{n-1}) = {\sE}_i$ for some $i$. 
\end{lemma}

\begin{proof} Set $D := \varphi(\P^{n-1})$ with reduced structure. Since the Picard number of $\P^{n-1}$ is $1$, $D$ is projective and $\varphi$ is not constant, it follows that $\dim D = n-1$ and therefore $D$ is 
a Cartier divisor on $X$. 

Assume that $D \cap {\sE}_i = \emptyset$ for all $i$. Then $\nu(D) \subset \overline{X} \setminus {\rm Sing}\, (\overline{X})$. Since 
$$\pi^{0} : A^{0} := A \setminus \pi^{-1}({\rm Sing}\, (\overline{X})) \to \overline{X}^0 := \overline{X} \setminus {\rm Sing}\, (\overline{X})$$ 
induced from $\pi$ is \'etale, whereas $\P^{n-1}$ has no non-trivial finite \'etale cover, the fiber product $\P^{n-1} \times_{\overline{X}^0} A^{0}$, under $\pi^{0}$ and $\nu \circ \varphi$, is a disjoint union of $\P^{n-1}$. From each factor $\P^{n-1}$ of $\P^{n-1} \times_{\overline{X}^0} A^{0}$, we have a non-constant morphism $\P^{n-1} \to A^{0} \subset A$ via the natural morphism of the fiber product. However, this is impossible as any morphism from $\P^{n-1}$ to $A$ is constant by the universality of the Albanese morphism and the fact that the Albanese variety of $\P^{n-1}$ is a point by $H^1(\P^{n-1}, \sO_{\P^{n-1}}) = 0$.

Hence there is $i$, say $1$, such that $D \cap {\sE}_1 \not= \emptyset$. Then, $D|_{{\sE}_1}$ is either a non-zero effective Cartier divisor on ${\sE}_1$ or $D = {\sE}_1$. Thus
$$\dim\, D \cap {\sE}_1 \ge n-2 \ge 1$$ 
by $n \ge 3$. In particular, there is an irreducible curve $C \subset D \cap {\sE}_1$.  

Let $\overline{X}_1$ be the blow up of $\overline{X}$ at all $P_j$ with $j \ge 2$ except $P_1$. Then $\overline{X}_1$ is projective and we have a birational morphism 
$$\nu_1 : X \to \overline{X}_1$$ 
such that $X \setminus {\sE}_1 \simeq \overline{X}_1 \setminus \nu_1({\sE}_1)$ and $\nu_1({\sE}_1)$ is a point, say again $P_1$, by slight abuse of language. Let $H$ be a very ample line bundle on $\overline{X}_1$. Then, as $C \subset {\sE}_1$ and $\nu_1({\sE}_1)$ is a point, we have 
$$(C.\nu_1^*H)_X = ((\nu_1)_{*}C.H)_{\overline{X}_1} = 0.$$
Let $Q_0 \in C \subset D$ and $Q \in D$ and choose a connected curve $C_Q \subset D$ such that $Q_0, Q \in C_Q$. Since we have a surjective morphism $\P^{n-1} \to D$, it follows that any two curves on $D$ are proportional. Then $C_Q = aC$ in ${\rm N}_{1}\, (D)_{\Q} = {\rm N}^{1}(D)_{\Q}^{*}$ for some $a \in \Q$. As then
$$((\nu_1)_{*}(C_Q).H)_{\overline{X}_1} = (C_{Q}.\nu_1^*H)_X = (C_Q.\nu_1^*H|_D)_{D} = a(C.\nu_1^*H|_D)_D = a(C.\nu_1^*H)_X = 0,$$ 
it follows that $\nu_1(C_Q)$ is a point. Hence 
$$\nu_1(Q) = \nu_1(C_Q) = \nu_1(Q_0) = \nu_1(C) = \nu_1({\sE}_1) = P_1,$$
and thefore $D \subset {\sE}_1$. Hence $D= \sE_1$ as both are irreducible of the same dimension. This completes the proof.
\end{proof}

\begin{remark}\label{rem22a} Not only the proof but also the statement of Lemma \ref{lem22} is not valid when $n=2$. Indeed, the image of the diagonal $\Delta \subset A_k^2$ under the natural map $A_k^2 \dasharrow  X_k^2$ is $\P^1$ which does not coincide any $\sE_i = \P^1$. In fact, $X_k^2$ has infinitely many rational curves, as $X_k^2$ is either a rational surface or a Kummer K3 surface of product type.  
\end{remark}

Let $f \in {\rm End}\, (X)$. Our first goal is to show that $f \in \Aut (X)$.

\begin{lemma}\label{lem23} $f({\rm Exc}\, (\nu)) \subset {\rm Exc}\, (\nu)$.
\end{lemma}
\begin{proof} $f(\sE_i)$ is not a point by Lemma \ref{lem21} (2). Since ${\sE}_i \simeq \P^{n-1}$, we may apply Lemma \ref{lem22} for $f|_{{\sE}_i} : {\sE}_i \to X$ to get the result.
\end{proof}

Let $\overline{f} := \nu \circ f \circ \nu^{-1}$. Then $\overline{f}$ is a dominant rational map from $\overline{X}$ to $\overline{X}$.

\begin{lemma}\label{lem23a} 
$\overline{f} \in {\rm End}\, (\overline{X})$. 
\end{lemma}
\begin{proof} By definition, $\overline{f}$ is a morphism on $\overline{X} \setminus {\rm Sing}\, \overline{X}$. By Lemma \ref{lem23}, for each $P_i \in {\rm Sing}\, \overline{X}$, there is $j$ such that $f(\nu^{-1}(P_i)) = \sE_j$. Then 
$$\overline{f}(P_i) = \nu\circ f \circ \nu^{-1}(P_i) = \nu(\sE_j) = P_j.$$
Hence the rational map $\overline{f}$ is actually a morphism, as $\overline{X}$ is normal.  \end{proof} 

\begin{lemma}\label{lem23b} 
$f^{-1}({\rm Exc}\, (\nu)) \subset {\rm Exc}\, (\nu)$.
\end{lemma}
\begin{proof} Otherwise, there is an irreducible component, say $D$, of $f^{-1}({\rm Exc}\, (\nu))$ such that $D \not\subset {\rm Exc}\, (\nu)$. Here $D$ is of dimension $n-1$ as $f^{-1}({\rm Exc}\, (\nu))$ is the support of the Cartier divisor $f^{*}({\rm Exc}\, (\nu))$. Thus $\nu(D)$ is of dimension $n-1 \ge 2$ by $D \not\subset {\rm Exc}\, (\nu)$. However, then we would have  
$$\overline{f} (\nu (D)) = \nu \circ f(D) \subset \nu ({\rm Exc}\, (\nu)) = \cup_{i=1}^{m} P_i.$$
This contradicts to Lemma \ref{lem21} (2). Hence the results follows.
\end{proof}

Since we do {\it not} yet know $\overline{f} \in \Aut \overline{X}$, the following lemma is not trivial: 
\begin{lemma}\label{lem24} 
$\overline{f}({\rm Sing}\, (\overline{X})) = {\rm Sing}\, (\overline{X})$ and 
$\overline{f}^{-1}({\rm Sing}\, (\overline{X})) = {\rm Sing}\, (\overline{X})$.
\end{lemma}
\begin{proof} Let $S := \{\sE_i\}_{i=1}^{m}$ be the set of irreducible components of ${\rm Exc}\, (\nu)$. By Lemmas \ref{lem23} and \ref{lem23b}, the correspondences $f$ and $f^{-1}$ from $X$ to $X$ induce correspondences $f|_S$ and $f^{-1}|_S$ from $S$ to $S$. By Lemma \ref{lem23b}, $f|_S$ is a surjective morphism. Thus both $f|_S$ and $f^{-1}|_S$ are bijective morphisms, as $S$ is a finite set.  Since $\overline{f} = \nu \circ f \circ \nu^{-1}$, it follows that $\overline{f}|_{{\rm Sing}\, (\overline{X})}$ and $\overline{f}^{-1}|_{{\rm Sing}\, (\overline{X})}$ are bijective morphisms from ${\rm Sing}\, (\overline{X})$ to ${\rm Sing}\, (\overline{X})$, which are iniverse to each other. \end{proof}

\begin{lemma}\label{lem25} Let $f \in {\rm End}\, (X)$. Let $\overline{f} \in {\rm End}\, (\overline{X})$ be as in Lemma \ref{lem23a}. Then 
\begin{enumerate}
\item There is $f_A \in {\rm Aut}\, (A)$ such that $\pi \circ f_A = \overline{f} \circ \pi$.
\item $\overline{f} \in \Aut (\overline{X})$ and $f \in \Aut (X)$. Moreover, $\Aut (X) = \Aut (\overline{X})$ under the identification ${\rm Bir}\, (X) = {\rm Bir}\, (\overline{X})$ by $\nu$. 
\end{enumerate}
\end{lemma}
\begin{proof} Since $\pi : A \to \overline{X}$ is Galois and the Galois group $\langle \sigma \rangle$ of $\pi$ does not contain non-trivial translations of $A$, the abelian variety $A$ is the Albanese closure of $\overline{X}$ (\cite[Lemma 8.1 and its proof]{CMZ20}). Therefore $\overline{f} \in {\rm End}\, (\overline{X})$ lifts to some $f_A \in {\rm End}\, (A)$, that is, there is an endomorphism $f_A$ such that $\pi \circ f_A = f \circ \pi$, by \cite[Lemma 8.1]{CMZ20}. Here we note that $f_A$ is separable as so are $\overline{f}$ and $\nu$. Indeed, for a finite field extension $L_1 \subset L_2 \subset L_3$, the extension $L_1 \subset L_3$ is separable if and only if so are $L_1 \subset L_2$ and $L_2 \subset L_3$.

Set $N := \pi^{-1}({\rm Sing}\, \overline{X})$. Then $N$ is a finite set and $\pi(N) = {\rm Sing}\, \overline{X}$ as $\pi$ is finite and surjective.
Since $\overline{f}^{\pm 1} ({\rm Sing}\, \overline{X}) =  {\rm Sing}\, \overline{X}$ by Lemma \ref{lem24}, it follows that $f_A^{\pm 1} (N) = N$. Since $f_A$ is separable and therefore necessarily finite and \'etale by Remark \ref{rem22} 
(1), it follows that $f_A$ is also of degree $1$ by $f_A^{\pm 1}(N) = N$. Therefore $f_A \in \Aut (A)$ by the Zariski main theorem. Thus, the finite separable surjective morphism $\overline{f}$ is also of degree one. Hence $\overline{f} \in \Aut (\overline{X})$ again by the Zariski main theorem, as $\overline{X}$ is normal. By Lemma \ref{lem23a}, we have $\Aut (X) \subset \Aut (\overline{X})$ by $\nu$. On the other hand, since $\nu : X \to \overline{X}$ is the blow up at the maximal ideals of ${\rm Sing}\, \overline{X}$, we have $\Aut (\overline{X}) \subset \Aut (X)$ by $\nu$. Thus $\Aut (X) = \Aut (\overline{X})$ as claimed. Therefore $f \in \Aut (X)$. This completes the proof of Lemma \ref{lem25}. \end{proof}

{\it Completion of the proof of Theorem \ref{mainthm1}.} Now we are ready to complete the proof of Theorem \ref{mainthm1}. By Lemma \ref{lem25} (2), we have that 
$${\rm End}\, (X) = \Aut (X) \simeq \Aut (\overline{X})$$ 
and that $\Aut (\overline{X})$ is isomorphic to $N(\langle \sigma \rangle) /\langle \sigma \rangle$, where $N(\langle \sigma \rangle) \subset {\rm Aut}\, (A)$ is the normalizer of $\langle \sigma \rangle$ in $\Aut (A)$. Thus we obtain the second statement of Theorem \ref{mainthm1} readily from the fact that 
$$\Aut (A) = A \rtimes {\Aut}_{{\rm group}}\, (A)$$ 
for an abelian variety $A$ \cite[Page 43, Corollary 1]{Mm70} and $\langle \sigma \rangle$ is in the center of ${\Aut}_{{\rm group}}\, (A)$ (by the shape of $\sigma$). This completes the proof of Theorem \ref{mainthm1}.

\section{Proof of Theorem \ref{mainthm2}}\label{sect3}

In this section, we show Theorem \ref{mainthm2}. We assume that the base field is $K = \overline{\Q}$ unless stated otherwise.

\subsection{Brief review on Kawaguchi-Silverman Conjecture (KSC)}\label{subsect31}
\hfill

We briefly recall some notion, basic tools and known facts on Kawaguchi-Silverman Conjecture (KSC), relevant to us. We restrict ourselves for ${\rm End}\, (V)$ of a normal projective variety $V$. 

Let $h_{\P^N} : \P^N = \P^N(\overline{\Q}) \to \R_{\ge 0}$ be the logarithmic Weil height (\cite[B1-B4, Part B]{HS00} or \cite[Appendix II]{Mm70} for the definition and basic properties). $h_{\P^N}$ depend on the choice of the projective linear coordinate of $\P^N$ but only up to bounded functions. If $H$ is an ample line bundle on $V$, then one has then an embedding 
$$V \subset |mH|^{*} = \P^{N}\,\, {\rm with}\,\, N = \dim |mH|$$
for large $m$. 
By adding suitable constant on $h_{\P^N}$, we may and will assume that $h_{\P^N}(x) \ge m$ and define $h_H$ by 
$$h_{H} : V = V(\overline{\Q}) \to \R_{\ge 0}\,\, ;\,\, x \mapsto \frac{h_{\P^{N}}(x)}{m}\,\, (\ge 1).$$
$h_H$ is defined up to bounded functions and we always choose $h_H$ so that $h_H(x) \ge 1$ for all $x \in V$.
\begin{definition}\label{def31} Let $V$ be a normal projective variety of dimension $d$ over $\overline{\Q}$ and $H$ an ample line bundle on $V$. Let $f \in {\rm End}\, (V)$. 
\begin{enumerate}
\item The first dynamical degree $d_1(f)$  is defined by
$$d_1(f) := \lim_{n \to \infty} ((f^n)^*H.H^{d-1})_V^{\frac{1}{n}} \ge 1.$$
\item The arithmeic degree $a_f(x)$ for $x \in X$ is defined by
$$a_f(x) := \lim_{n \to \infty} h_H(f^n(x))^{\frac{1}{n}} \ge 1.$$
\end{enumerate}
\end{definition}

We have the following:
\begin{theorem}\label{thm31} Let $V$ be a normal projective variety and $f \in {\rm End}\, (V)$. Then:
\begin{enumerate}
\item $d_1(f)$ is well-defined and satisfies $d_1(f) \ge 1$. Moreover, $d_1(f)$ coincides with the spectral radius of $f^*|_{N^1(V)_{\Q}}$. 
\item $a_f(x)$ ($x \in X$) is well-defined and $a_f(x) \ge 1$. Moreover, $a_f(x)$ does not depends on the choice of $H$, $m$ 
and bounded functions added to $h_H$.
\item $1 \le a_f(x) \le d_1(f)$ for all $x \in X$.
\end{enumerate}
\end{theorem}
The assertions (1) is well-known. The proof of \cite[Theorem 1.1]{Tr15} is valid without any change for a normal projective variety $V$ and an ample class $H$ if $f \in {\rm End}\, (V)$. The assertions (2) and (3) are due to Kawaguchi-Silverman \cite[Theorem 3]{KS16a} and \cite{Ma20} (See also a survey paper \cite{Ma23}).

The following known results are relevant for us:

\begin{theorem}\label{thm32} Let $V$ be a normal projective variety and $f \in {\rm End}\, (V)$. Then:
\begin{enumerate}
\item KSC holds for ${\rm End}\, (V)$ when $V$ is an abelian variety (\cite[Theorem 4]{KS16a}, \cite[Theorem 1.2]{Si17}).
\item KSC holds for ${\rm End}\, (V)$ when $\dim V \le 2$ (\cite[Theorem 5]{KS16b} when $\dim\, V=1$, and \cite[Theorem 1.3]{MZ22}, see also \cite{Ka08} and 
\cite{MSS18}, when $\dim V = 2$).
\item KSC holds for ${\rm End}\, (V)$ when $V$ is a smooth rationally connected variety with at least one inter amplified endomorphism (\cite[theorem 1.1]{MY22}, see also \cite{MZ22}, \cite{Yo21}).
\end{enumerate}
\end{theorem}

The following proposition or its variants are also known (\cite{Si17}, \cite{MSS18}, \cite{LS21}, \cite{CLO22}, \cite{Ma23} for several variants):

\begin{proposition}\label{prop31} Let $V$ be a normal projective variety and $f \in \Aut (V)$.
Let $\nu : V \to W$ be a surjective morphism to 
a normal projective variety $W$. Assume that there is $g \in \Aut (W)$ such that $g \circ \nu = \nu \circ f$. 
\begin{enumerate}
\item Assume that $\dim V = \dim W$ or $\dim V = \dim W +1$. Then, if KSC holds for $(W,g)$, then the KSC also holds for $(V,f)$.
\item Assume that $\dim V = \dim W$, that is, the morphism $\nu$ is generically finite. Then, if KSC holds for $(V,f)$, then the KSC also holds for $(W,g)$.
\end{enumerate}
\end{proposition}

\begin{proof} The assertions (1) and (2) are special cases of \cite[Lemmas 3.1 and 3.2]{CLO22}. Smoothness assumption there is not needed here as we assume that $f$, $g$, $\nu$ are morphisms here.
\end{proof}

Let $n$ be an integer such that $n \ge 2$ and let 
$M_{n} \in {\rm SL}\, (n, \Z)$
whose characteristic polynomial is
the minimal polynomial of a Pisot unit $a_n >1$ 
such that $\Q(a_n) = \Q(\sqrt[n]{2})$ (See \cite[Theorem 3.3]{Og19} and references therein). Recall that $A_{k}^{n} = E_{k}^{\times n}$ from Defintion \ref{def1}. Then we have a natural action ${\rm GL}\, (n, \Z)$ on the group variety $A_{k}^{n} = E_{k}^{\times n}$ and therefore an embedding ${\rm GL}\, (n, \Z) \subset \Aut_{{\rm group}} (A_{k}^{n})$. We set $f_n := M_{n} \in \Aut_{{\rm group}} (A_{k}^{n})$.

Then by \cite[Proposition 6.2]{CLO22}, we have the following:

\begin{proposition}\label{prop32}
$d_1(f_n) = a_n^2 >1$ and there is $z \in A_{k}^{n}$ such that $O_{f_n}(z)$ is Zariski dense in $A_{k}^{n}$.
\end{proposition}

\subsection{Proof of Theorem \ref{mainthm2}}\label{subsect32}
\hfill

Now we are ready to prove Theorem \ref{mainthm2}. Let $n \ge 3$ and $X := X_k^n$ be a variety of Ueno type. For ease of notation, we set $E := E_k$, $\sigma := \sigma_k^n$, $A = A_k^n = E^{\times n}$, $\overline{X} = \overline{X}_k^n$, $\overline{V}$, $\pi := \pi_k^n$, $\nu := \nu_k^n$ and $\mu := \mu_k$ the multiplicative group of order $k$. Then, under
$$X \xrightarrow{\nu} \overline{X} = A/\mu \xleftarrow{\pi} A,$$ 
any $f \in \Aut (X)$ descends to $\overline{f} \in \Aut (\overline{X})$ and $\overline{f}$ lifts to $f_A \in \Aut (A)$ by Lemma \ref{lem25}. 

Let $f \in \Aut(X)$. Recall that KSC hold for $f_A \in \Aut (A)$ by Theorem \ref{thm32} (1). Since $\pi$ is a finite surjective morphism, KSC holds for $\overline{f} \in \Aut(\overline{X})$ by Proposition \ref{prop31} (2). Since $\nu$ is birational morphism, KSC then holds for $f \in \Aut(X)$ by Proposition \ref{prop31} (1). This proves the first statement.

Let $f_n \in \Aut (A)$ and $z \in A$ as in Proposition \ref{prop32}. Then $f_n$ commutes with $\pi$ so that it descends to $f_{\overline{X}} \in \Aut (\overline{X})$ equivariantly with respect to $\pi$ such that $O_{f_{\overline{X}}}(\pi(z))$ is Zariski dense in $\overline{X}$. By Lemma \ref{lem25} (2), it follows that the automorphism $\overline{f}$ of $\overline{X}$ lifts to an automorphism of $X$ equivariantly with respect to $\nu$, which we denote by $f \in \Aut (X)$. By replacing $\pi(z)$ by $f_{\overline{X}}^m(\pi(z))$ for suitable $m$, we may and will assume that $\pi(z) \not\in {\rm Sing}\, (\overline{X})$. Then for the unique $x \in X$ such that $\nu(x) = \pi(z)$, the set $O_{f}(x)$ is Zariski dense in $X$. Moreover, since the dynamical degree is invariant under an equivariant generically finite map by \cite[Theorem 1.1]{DN11}, it follows that $d_1(f) = d_1(f_n) > 1$. Hence this $f \in \Aut (X)$ satisfies KSC non-trivially and non-vacuously. This completes the proof of Theorem \ref{mainthm2}.

\begin{remark}\label{rem31} 
By Theorem \ref{thm32} (2), KSC for ${\rm End}\, (X_{k}^{n})$ also holds for $n \le 2$.  
\end{remark}

\section{Proof of Theorem \ref{mainthm3}}\label{sect4}

\subsection{$c_2$-contractions on a Calabi-Yau threefold - Revisited}\label{subsect41}
\hfill

In this section, we briefly recall basic facts on $c_2$-contraction on a Calabi-Yau threefold from \cite{OS01} and \cite{Og24}. See also Introduction \ref{subsect13} for the definition of Calabi-Yau threefold, $c_2$-contraction on a Calabi-Yau threefold. We call a $c_2$-contraction $f : X \to W$ maximal if for any $c_2$-contraction $f' : X \to W'$ factors through $f$. Then $\Aut (X)$ preserves $f$. Our base field is either $\overline{\Q}$ or $\C$. All statements below are valid over both fields unless specified otherwise. Note that ${\rm End}\, (X) = \Aut (X)$ for a Calabi-Yau threefold $X$, by Remark \ref{rem22} (1) and by the simply connectedness assumption. 

The following two examples (\cite{Og96}, see also \cite[Theorem 3,7]{OS01}) play the most important role in this Section. 

\begin{example}\label{ex41}
\begin{enumerate}
\item Under the notation in Definition \ref{def1} (3), 
$$\nu_3 : X_3 := X_3^3 \to \overline{X}_3 := \overline{X}_3^3$$ 
is a birational $c_2$-contraction of a Calabi-Yau threefold $X_3$. Here we recall that $\overline{X}_3$ is the quotient variety of the abelian variety $A_3 := E_{\zeta_3}^{\times 3}$ by $\mu_3 = \langle \sigma_3^3 \rangle$:
$$\pi_3 : A_3 \to A_3/\mu_3 = \overline{X}_3.$$

\item Let $\zeta_7 := {\rm exp}(2\pi\sqrt{-1}/7)$ and $A_7 := {\rm Alb}\, (C)$ the Albanese variety of the Klein quartic curve
$$C := \{[x:y:z]\,|\, xy^3+yz^3+zx^3= 0\} \subset \P^2.$$
$A_7$ is an abelian threefold. Let $g_7 \in \Aut_{{\rm group}} (A_7)$ be the automorphism induced from the automorphism $g \in \Aut (C)$ given by 
$$g([x : y : z]) = [\zeta_7 x: \zeta_7^2 y :\zeta_7^4 z].$$ 
Then $g_7$ is of order $7$ and the action of $g_7$ on $H^0(\Omega_{A_7}^1)$ is 
represented by
$$\left(
    \begin{array}{ccc}
      \zeta_7 & 0 & 0\\
      0 & \zeta_7^2 & 0\\
      0 & 0 & \zeta_7^4
\end{array}
  \right)$$
under suitable basis of $H^0(\Omega_{A_7}^1)$. Let $\mu_7 = \langle g_7 \rangle$. Let $$\pi_7 : A_7 \to \overline{X}_7 := A_7/\mu_7$$ 
be the quotient morphism and 
$$\nu_7 : X_7 \to \overline{X}_7 := A_7/\mu_7$$ 
be a crepant projective resolution of $\overline{X}_7$, which is unique (\cite{Og96}). Then $X_7$ is a Calabi-Yau threefold and $\nu_7$ is a $c_2$-contraction of a Calabi-Yau threefold $X_7$. Note that the automorphism of order $3$ on $C$ defined by $[x:y:z] \mapsto [z:x:y]$ induces an automorphism $\sigma \in \Aut_{{\rm group}} (A_7)$ of order $3$ such that $\sigma^{-1} \circ g_7 \circ \sigma = g_7^2$.
\end{enumerate}
\end{example}

By the explicit descriptions in Example \ref{ex41}, both $X_3$ and $X_7$ are defined not only over $\C$ but also over $\overline{\Q}$. We cite the following from \cite{Og96} (see also \cite[Therem 3.4]{OS01}):

\begin{theorem}\label{thm41} Let $f : X \to W$ be a birational $c_2$-contraction of a Calabi-Yau threefold $X$. Then $f : X \to W$ is isomorphic either $\nu_3 : X_3 \to \overline{X}_3$ or $\nu_7 : X_7 \to \overline{X}_7$. Moreover, $f$ is a maximal $c_2$-contraction of $X$ in each case.
\end{theorem} 

In order to prove Theorem \ref{mainthm3}, we also need the following structure theorem of $\Aut(X_i)$ ($i=3$ and $7$) in Theorem \ref{thm41} and Example \ref{ex41}, which may have its own interest:

\begin{theorem}\label{thm42} Let $X_3$ and $X_7$ be as in Example \ref{ex41}. Then, under the notation of Example \ref{ex41}, the following (1) and (2) hold: 
\begin{enumerate}
\item For $X_3$, we have
$$\Aut (X_3) = A_3^{\mu_3} \rtimes \Aut_{{\rm group}}(A_3)/\mu_3 \simeq (\Z/3)^{\oplus 3} \rtimes {\rm GL}\, (3, \Z[\zeta_3])/\mu_3,$$ 
and there is $g \in \Aut\, (X_3)$ such that $d_1(g) > 1$ and $O_g(x)$ for some $x \in X_3$ is Zariski dense in $X_3$ both over $\overline{\Q}$ and over $\C$. 
\item For $X_7$, we have 
$$\Aut (X_7) \simeq (\Z/7) \rtimes (\Z[\zeta_7]^{\times} \rtimes \Z/3)/\mu_7 \simeq (\Z/7) \rtimes ((\Z/2 \oplus \Z^{\oplus 2}) \rtimes \Z/3).$$ 
In particular $\Aut (X_7)$ is a virtually abelian group, that is, $\Aut (X_7)$ has a finite index abelian subgroup. There is also $g \in \Aut\, (X_7)$ such that $d_1(g) > 1$ and $O_g(x)$ for some $x \in X_7$ is Zariski dense in $X_7$ both over $\overline{\Q}$ and over $\C$. 
\end{enumerate}
\end{theorem} 

\begin{proof} The assertion (1) follows from Theorem \ref{mainthm1} and Proposition \ref{prop32}. 

Let us show the assertion (2). Since $\nu_7 : X_7 \to \overline{X}_7$ is a maximal $c_2$-contraction and $\nu_7$ is also a unique projective crepant resolution of $\overline{X}_7$, we have $\Aut (X_7) \simeq \Aut (\overline{X}_7)$ under $\nu_7$. Since $\mu_7$ contains no non-trivial translation, it follows that $\pi_7 : A_7 \to \overline{X}_7$ is the Albanese reduction \cite[Lemma 8.1]{CMZ20}. In particular, any element of $\Aut (\overline{X}_7)$ lifts to $\Aut (A_7)$ equivariantly with respect to $\pi_7$ by \cite[Lemma 8.1]{CMZ20}. Hence $\Aut (\overline{X}_7)$ is isomorphic to $N(\mu_7)/\mu_7$, where $N(\mu_7)$ is the normalizer of $\mu_7$ in $\Aut (A_7)$. Let $C(\mu_7)$ is the centralizer of $\mu_7$ in $\Aut (A_7)$. We also denote by $N_{{\rm group}}(\mu_7)$ (resp. $C_{{\rm group}}(\mu_7)$) the normalizer (resp. centralizer) of $\mu_7$ in $\Aut_{{\rm group}} (A_7)$. 

\begin{lemma}\label{lem41} $N(\mu_7) \simeq \Z/7 \rtimes N_{{\rm group}}(\mu_7) = \Z/7 \rtimes (C_{{\rm group}}(\mu_7) \rtimes \Z/3)$. 
\end{lemma}

\begin{proof} Let $h \in N(\mu_7)$ and decompose $h$ as $h = t_a \circ h_0$, where $h_0 \in {\rm Aut}_{{\rm group}}(A_7)$ and $t_a$ is the translation by $a \in A_7$. Then for $x \in A_7$ we have 
$$h^{-1} \circ g_7 \circ h (x) =  h_0^{-1}(g_7(h_0(x)) + g_7(a) -a) = h_0^{-1} \circ g_7 \circ h_0 (x) + h_0^{-1}(g_7(a) -a),$$
and therefore
$$h^{-1} \circ g_7 \circ h = t_{h_0^{-1}(g_7(a) -a)} \circ (h_0^{-1} \circ g_7 \circ h_0).$$
As the left hand side is in $\mu_7$, it follows that  $h_0^{-1}(g_7(a) - a) = 0$ and $h_0^{-1} \circ g_7 \circ h_0 \in \mu_7$. From $h_0^{-1}(g_7(a) - a) = 0$, we have $g_7(a) = a$. This means that $a \in A_7^{\mu_7} \simeq \Z/7$ and we obtain $N(\mu_7) \simeq \Z/7 \rtimes N_{{\rm group}}(\mu_7)$. 

It remains to show that $N_{{\rm group}}(\mu_7) \simeq  C_{{\rm group}}(\mu_7) \rtimes \Z/3$.

Note that $\Aut_{{\rm group}} (\mu_7) \simeq \Z/6$. Combining $h_0^{-1} \circ g_7 \circ h_0 \in \mu_7$ with the fact that $h_0^{-1} \circ g_7 \circ h_0$ and $g_7$ has the same eigenvalues on $H^0(\Omega_{A_7}^1)$, it follows that $h_0^{-1} \circ g_7 \circ h_0$ is either one of $g_7$, $g_7^2$, $g_7^4$. Hence $h_0^3 \circ g_7 = g_7 \circ h_0^3$. Since the action on $H^0(\Omega_{A_7}^1)$ of $\Aut_{{\rm group}}(A_7)$ is faithful, it follows that $h_0^{3} \in C_{{\rm group}}(\mu_7)$. Moreover, since we have $\sigma^{-1} \circ g_7 \circ \sigma = g_7^2$ and $\sigma$ is of order $3$, where $\sigma$ is the automorphism in Example \ref{ex41} (2), it follows that 
$$N_{{\rm group}}(\mu_7) = C_{{\rm group}}(\mu_7) \rtimes \langle \sigma \rangle \simeq C_{{\rm group}}(\mu_7) \rtimes \Z/3,$$
as claimed.
\end{proof}

Next we compute $C_{{\rm group}}(\mu_7)$.

\begin{lemma}\label{lem42} Under $\mu_{14} \simeq \Z/14$, we have
$$C_{{\rm group}}(\mu_7) \simeq \Z[\zeta_7]^{\times} \simeq \Z/14 \oplus \Z^{\oplus 2}.$$
\end{lemma}

\begin{proof} We regard $\overline{\Q} \subset \C$ and denote by $(A_7)_{\overline{\Q}}$ and $(A_7)_{\C}$) if $A_7$ is considered as a $\overline{\Q}$-variety and a $\C$-variety respectively. Since $\Aut_{{\rm group}} ((A_7)_{\overline{\Q}}/\overline{\Q})$ is locally of finite type and of dimension $0$ as a scheme and $\overline{\Q}$ is algebraically closed, by the Hilbert zero point theorem, the group of $\C$-valued points is the same as the group of $\overline{\Q}$-valued points. Therefore
$$\Aut_{{\rm group}} ((A_7)_{\overline{\Q}}/\overline{\Q}) = \Aut_{{\rm group}} ((A_7)_{\C}/\C),$$ 
as abstract groups. So, it suffices to prove Lemma \ref{lem42} 
for $(A_7)_{\C}$. 

In what follows, $A_7 = (A_7)_{\C}$ and also consider $A_7$ as a complex torus. Then, under the identification $\langle g_7 \rangle = \langle \zeta_7 \rangle$ given by $g_7 \leftrightarrow \zeta_7$, we have 
$$\Z[\zeta_7]^{\times} \subset \Aut_{{\rm group}} (A_7)\,\, ,\,\, {\rm hence}\,\, ,\,\,  \Z[\zeta_7]^{\times} \subset C_{{\rm group}}(\mu_7).$$ 
On the other hand, by \cite{Og96}, we have also $H^1(A_7, \Z) \simeq \Z[\zeta_7]$ as $\Z[\zeta_7]$-modules with $\Z[\zeta_7]$-action on the left hand side is given by the identification $g_7 = \zeta_7$ above. Since $\Aut_{{\rm group}} (A_7)$ acts faithfully on $\Z$-module $H^1(A_7, \Z)$, it follows that $C_{{\rm group}}(\mu_7)$ also acts faithfully on $H^1(A_7, \Z)$. The action is not only as $\Z$-module but also as $\Z[\zeta_7]$-module by the definition of $C_{{\rm group}}(\mu_7)$. Thus 
$$C_{{\rm group}}(\mu_7) \subset \Z[\zeta_7]^{\times}$$
as well. Hence $C_{{\rm group}}(\mu_7) = \Z[\zeta_7]^{\times}$ as claimed. 

Let us compute the abelian group $\Z[\zeta_7]^{\times}$. If $x \in \Z[\zeta_7]^{\times}$ is of finite order, then  $x$ is a root of $1$ in $\Z[\zeta_7]$. Thus the torsion subgroup of $\Z[\zeta_7]^{\times}$ is $\mu_{14}$. Since $\Q(\zeta_7)$ has exactly $6$ different embeddings to $\C$ and no embedding to $\R$, by the Dirichlet unit theorem, the rank of $\Z[\zeta_7]^{\times}$ is $6/3 -1 = 2$. Hence $\Z[\zeta_7]^{\times} \simeq \Z/14 \oplus \Z^{\oplus 2}$ as claimed. This completes the proof of Lemma \ref{lem42}.
\end{proof}

By Lemmas \ref{lem41} and \ref{lem42}, we obtain the first two assertions of (2). 

It remains to show the last assertion of (2). Setting $\id_{A_7} = \id$, 
we have 
$$f := g_7 + \id \in {\rm End}_{{\rm group}}(A_7).$$
Then $f \in C_{{\rm group}} (\mu_7)$, as 
$$(g_7+ \id)(g_7^5 +g_7^3 +g_7) = g_7^6 + g_7^5+g_7^4+g_7^3+g_7^2 + g_7 = -\id\,\, ,\,\, (g_7+ \id) \circ g_7 = g_7 \circ (g_7+ \id)$$
in ${\rm End}_{{\rm group}}(A_7)$. 

Thus $f$ descends to $g \in {\rm Aut}\, (X_7)$.

The next lemma completes the proof of the last statement of (2), as the natural map $A_7 \dasharrow X_7$ is dominant and generically finite and $d_1(g) = d_1(f)$ by \cite[Theorem 1.1]{DN11}.

\begin{lemma}\label{lem43}
$d_1(f) > 1$ and $f$ has a Zariski dense orbit on $A_7$.
\end{lemma}

\begin{proof} 
Since the eigenvalues of $f$ on $H^0(\Omega_{A_7}^1)$ are 
$$\zeta_7 +1\,\, ,\,\, \zeta_7^2 +1\,\, ,\,\, \zeta_7^4+1,$$
it follows that the eigenvalues of the action of $f$ on $H^{1,1}(A_7)$ are
$$(1+\zeta_7^a)\overline{(1+\zeta_7^b)}\,\, {\rm with}\,\, a, b \in \{1,2, 4\}$$Therefore, 
$$d_1(f) =  |\zeta_7 +1|^2 >1.$$
In particular, $f$ is of infinite order. Since $1$ is not an eigenvalue of $f^*$ on $H^{1,1}(A_7)$ and on $H^{2,2}(A_7)$ by the description above, it follows that $f$ admits no invariant divisors and curves. Thus the Zariski closure of $O_f(x)$ ($x \in A_7$) has to be either a finite set of points or the whole $A_7$. Since $f$ is of infinite order, it follows then that the Zariski closure of $O_f(x)$ for some $x \in A_7$ is the whole $A_7$, as $O_f(x)$ is an infinite set for general $x \in A_7$ even over $\overline{\Q}$ by Amerik \cite{Am11}. This completes the proof of Theorem \ref{thm41}.   
\end{proof}

This completes the proof of Theorem \ref{thm42}.
\end{proof}

\subsection{Completion of the proof of Theorem \ref{mainthm3}}\label{subsect42}
\hfill

We are now ready to prove Theorem \ref{mainthm3}. As remarked at the beginning of Subsection \ref{subsect41}, we know that ${\rm End}\, (X) = \Aut\, (X)$ for a Calabi-Yau threefold $X$. Let $X$ be a Calabi-Yau threefold defined over $\overline{\Q}$ with a non-trivial $c_2$-contraction and $f : X \to W$ a maximal $c_2$-contraction of $X$. 
Then $\dim\, W = 1$, $2$ or $3$. Our proof of the cases $\dim\, W = 1$, $2$ below is the same as in \cite[Theorem 1.8]{LS21}. These two cases also follow from \cite[Theorem 1.4]{Li24}.

Let $g \in \Aut (X)$ and $g_W \in \Aut (W)$ induced from $g$ via the maximality of $f$.

Consider first the case where $\dim\, W = 1$. Then $W = \P^1$ and by \cite{VZ01}, the critical set $N \subset \P^1$ of $f$ satisfies $|N| \ge 3$. Then $g_W^{|N|!}$ is identity on $N$ and therefore $g_W^{|N|!}$ is identity on $\P^1$ by $|N| \ge 3$. Hence $g$ has no Zariski dense orbit and KSC holds vacuously in this case. 

Consider next the case where $\dim\, W = 2$. Then $(W, g_W)$ satisfies KSC as $\dim W = 2$ by Theorem \ref{thm32} (2). Thus $(X, g)$ satisfies KSC as well by Proposition \ref{prop31} (1). 

Consider the case where $\dim W = 3$. Then we have an identification $X = X_i$ ($i=3$, $7$) and morphisms
$$X_i \xrightarrow{\nu_i} \overline{X}_i = A_i/\mu_i \xleftarrow{\pi_i} A_i$$
as in Example \ref{ex41}. Let $h \in \Aut (X_i)$. Then $h$ descends to $\overline{h} \in \Aut (\overline{X})$ by the maximality of $\nu_i$ and lifts to $\tilde{h} \in \Aut (A_i)$, as $A_i$ is the Albanese closure of $\overline{X}_i$ (\cite[Lemma 8.1]{CMZ20}). Since KSC holds for $(A_i, \tilde{h})$ by Theorem \ref{thm32} (1), it follows that KSC holds for $(\overline{X}_i, \overline{h})$ by Proposition \ref{prop31} (2). Thus, KSC also holds for $(X_i, h)$ by Proposition \ref{prop31} (1). KSC also holds non-trivially and non-vacuously for $\Aut (X_i)$ by Theorem \ref{thm42}. This completes the proof of Theorem \ref{mainthm3}.



\begin{thebibliography}{CCCCCC}

\bibitem[Am11]{Am11}
E. Amerik. :  \textit{Existence of non-preperiodic algebraic points for a rational self-map of infinite order}, Math. Res. Lett. {\bf 18} (2011) 251--256.

\bibitem[Be83]{Be83} A. Beauville, : \textit{Some remarks on K\"ahler manifolds with $c_1 = 0$}, Classification of algebraic and analytic manifolds (Katata, 1982), 1--26, Progr. Math. {\bf 39} Birkh\"auser Boston, Boston, MA, 1983.
	
\bibitem[CMZ20]{CMZ20} P. Cascini, S. Meng, D.-Q. Zhang, : {\it Polarized endomorphisms of normal projective threefolds in arbitrary characteristic}, Math. Ann. {\bf 378} (2020) 637--665.

\bibitem[COV15]{COV15} F. Catanese, K. Oguiso, A., Verra, : {\it On the unirationality of higher dimensional Ueno-type manifolds}, Special Issue, Rev. Roumaine Math. Pures Appl. {\bf 60} (2015) 337--353.

\bibitem[CLO22]{CLO22} J.-A. Chen, H.-Y. Lin, K. Oguiso, : {\it On the Kawaguchi--Silverman Conjecture for birational automorphisms of irregular varieties}, arXiv:2204.09845.

\bibitem[CT15]{CT15} J.-L. Colliot-Th\'el\`ene, : {\it Rationalit\'e d'un fibr\'e en coniques}, Manuscripta Math. {\bf 147} (2015) 305--310.
		
\bibitem[DLOZ22]{DLOZ22}
T.-C. Dinh, H.-Y. Lin, K. Oguiso and D.-Q. Zhang, : \textit{Zero entropy automorphisms of compact K\"ahler manifolds and dynamical filtrations}, Geom. Funct. Anal. {\bf 32} (2022), no. 3, 568--594.
		
\bibitem[DN11]{DN11} 
T.-C. Dinh, V.-A. Nguy\^en, : \textit{Comparison of dynamical degrees for semi-conjugate meromorphic maps}, Comment. Math. Helv.  {\bf 86} (2011), no. 4, 817--840.

\bibitem[DS05]{DS05} T.-C. Dinh, N. Sibony, : {\it Une borne sup\'erieure pour l'entropie topologique d'une application rationnelle}, Ann. of Math. {\bf 161} (2005) 1637--1644.

\bibitem[Fu02]{Fu02} Y. Fujimoto, : {\it Endomorphisms of Smooth Projective
3-Folds with Non-Negative Kodaira Dimension}, Publ. RIMS, Kyoto Univ.
{\bf 38} (2002) 33--92. 

\bibitem[HS00]{HS00}
M. Hindry, J. H. Silverman, :  \textit{Diophantine geometry. An introduction. Graduate Texts in Mathematics}, {\bf 201}, Springer-Verlag, New York, 2000.

\bibitem[Ka08]{Ka08}
S. Kawaguchi, : \textit{Projective surface automorphisms of positive topological entropy from an arithmetic viewpoint}, Amer. J. Math. {\bf 130} (2008) 159--186.

\bibitem[KS16a]{KS16a}
S. Kawaguchi, J. H. Silverman, : \textit{Dynamical canonical heights for Jordan blocks, arithmetic degrees of orbits, and nef canonical heights on abelian varieties}, Trans. Amer. Math. Soc. {\bf 368} (2016) 5009--5035.

\bibitem[KS16b]{KS16b}
S. Kawaguchi, J. H. Silverman, : \textit{On the dynamical and arithmetic degrees of rational self-maps of algebraic varieties}, J. Reine Angew. Math. {\bf 713} (2016) 21--48.

\bibitem[KOZ09]{KOZ09} J.-H. Keum, K. Oguiso, D.-Q. Zhang, : {\it Conjecture of Tits type for complex varieties and theorem of Lie-Kolchin type for a cone}, Math. Res. Lett. {\bf 16} (2009) 133--148.

\bibitem[Ko96]{Ko96} J. Koll\'ar, : {\it Rational curves on algebraic varieties}, A Series of Modern Surveys in Mathematics, {\bf 32}, Springer-Verlag, Berlin, 1996.

\bibitem[KL09]{KL09} J. Koll\'ar, M. Larsen, : {\it 
Quotients of Calabi-Yau varieties}, Algebra, arithmetic, and geometry: in honor of Yu. I. Manin. Vol.II, 179--211,
Progr. Math., 270, Birkh\"auser Boston, Boston, MA, 2009.

\bibitem[LS21]{LS21}
J. Lesieutre, M. Satriano. \textit{Canonical heights on hyper-Kahler varieties and the Kawaguchi--Silverman conjecture}, Int. Math. Res. Not. IMRN (2021), no. 10 7677--7714.

\bibitem[LO09]{LO09} N.-H. Lee, K. Oguiso, : {\it Connecting certain rigid birational non-homeomorphic Calabi-Yau threefolds via Hilbert scheme}, Comm. Anal. Geom. {\bf 17} (2009) 283--303.

\bibitem[Li24]{Li24} S. Li, : {\it Kawaguchi-Silverman conjecture on automorphisms of projective threefolds}, International Journal of Mathematics (2024), doi:10.1142/S0129167X24500022, arxiv: 2209.06815.

\bibitem[Ma20]{Ma20}
Y. Matsuzawa, : \textit{On upper bounds of arithmetic degrees}, Amer. J. Math. {\bf 142} (2020) 1797--1820.

\bibitem[Ma23]{Ma23}
Y. Matsuzawa, : {\it Recent advances on Kawaguchi-Silverman conjecture}, arXiv:2311.15489, to appear in the proceedings for the Simons Symposia on Algebraic, Complex, and Arithmetic Dynamics.


\bibitem[MSS18]{MSS18}
Y. Matsuzawa, K. Sano, T. Shibata, : \textit{Arithmetic degrees and dynamical degrees of endomorphisms on surfaces}, Algebra Number Theory {\bf 12} (2018) 1635--57.

\bibitem[MY22]{MY22} Y., Matsuzawa, S. Yoshikawa, : {\it Kawaguchi-Silverman conjecture for endomorphisms on rationally connected varieties admitting an int-amplified endomorphism}, Math. Ann. {\bf 382} (2022) 1681--1704.



\bibitem[MZ22]{MZ22} S. Meng, D.-Q. Zhang, : {\it Kawaguchi-Silverman conjecture for certain surjective endomorphisms}, Doc. Math. {\bf 27} (2022) 1605--1642.

\bibitem[MZ23]{MZ23} S. Meng, D.-Q. Zhang, De-Qi, : {\it Advances in the equivariant minimal model program and their applications in complex and arithmetic dynamics}, arXiv:2311.16369, to appear in the proceedings for the Simons Symposia on Algebraic, Complex, and Arithmetic Dynamics.


\bibitem[Mi87]{Mi87} Y. Miyaoka, : {\it The Chern classes and Kodaira dimension of a minimal variety}, Algebraic geometry, Sendai, 1985, 449--476, Adv. Stud. Pure Math. (1987) {\bf 10}, North-Holland, Amsterdam.

     
\bibitem[Mm70]{Mm70} D.. Mumford,: {\it Abelian varieties}, Tata Institute of Fundamental Research Studies in Mathematics, 5. Published for the Tata Institute of Fundamental Research, Bombay by Oxford University Press, London, 1970. 


\bibitem[NZ09]{NZ09} N. Nakayama, D.-Q. Zhang, : {\it Building blocks of \'etale endomorphisms of complex
projective manifolds}, Proc. London Math. Soc. {\bf 99} (2009) 725--756.

\bibitem[Og96]{Og96} K. Oguiso: {\it On the complete classification of Calabi-Yau threefolds of type $III_{0}$}, In Higher-dimensional 
Complex Varieties (Trento, 1994), de Gruyter (1996) 329--339. 
		
\bibitem[Og14]{Og14} K. Oguiso, : \textit{Some aspects of explicit birational geometry inspired by complex dynamics}, Proceedings of the International Congress of Mathematicians?Seoul 2014. Vol. II, 695--721, Kyung Moon Sa, Seoul, 2014.

\bibitem[Og19]{Og19} K. Oguiso, : {\it Pisot units, Salem numbers, and higher dimensional projective manifolds with primitive automorphisms of positive entropy}, Int. Math. Res. Not. IMRN (2019) 1373--1400.

\bibitem[Og24]{Og24} K. Oguiso, : {\it Fibered Calabi-Yau threefolds with relative automorphisms of positive entropy and $c_2$-contractions}, preprint (2024). 

\bibitem[OS01]{OS01} K. Oguiso,  J. Sakurai : \textit{ Calabi-Yau threefolds of quotient type}, Asian J. Math. {\bf 5} (2001) 43--77.
	
\bibitem[OT15]{OT15} K. Oguiso, T.-T. Truong, : {\it Explicit examples of rational and Calabi-Yau threefolds with primitive automorphisms of positive entropy}, J. Math. Sci. Univ. Tokyo {\bf 22} (2015) 361--385.

\bibitem[Si17]{Si17}
J. H. Silverman.  \textit{Arithmetic and dynamical degrees on abelian varieties} , J. Theor. Nombres Bordeaux {\bf 29} (2017) 151--167.

\bibitem[Tr15]{Tr15} T.-T. Truong,: {\it (Relative) dynamical degrees of rational maps over an algebraic closed field}, arXiv:1501.01523.	
		
\bibitem[Tr20]{Tr20}  T.-T. Truong, : \textit{Relative dynamical degrees of correspondences over a field of arbitrary characteristic}, J. Reine Angew. Math. {\bf 758} (2020), 139--182. 
		
\bibitem[Ue75]{Ue75} K. Ueno, : \textit{Classification theory of algebraic varieties and compact complex spaces}. Lecture Notes in Mathematics, {\bf 439}, Springer-Verlag, 1975.


\bibitem[VZ01]{VZ01} E. Viehweg, K. Zuo, : {\it On the isotriviality of families of projective manifolds over curves}, J. Algebraic Geom. {\bf 10} (2001) 781--799.
		
\bibitem[Wi97]{Wi97} P. M. H. Wilson,: \textit{The role of c2 in Calabi-Yau classification -- a preliminary survey}, Mirror Symmetry, II, AMS/IP Stud. Adv. Math., vol. 1, Amer. Math. Soc., Providence, RI, 1997, pp. 381--392.

\bibitem[Yo21]{Yo21} S. Yoshikawa. Structure of Fano fibrations of varieties admitting an int-amplified endomorphism. Adv. Math., {\bf 391} (2021) Paper No.107964, 32.


\end{thebibliography}
\end{document}